\documentclass[notitlepage,11pt,reqno]{amsart}
\usepackage{amssymb,amsmath,amscd,xspace,amsthm,mathrsfs}
\usepackage[usenames]{color}
\usepackage[breaklinks=true]{hyperref}
\usepackage[letterpaper]{geometry}
\newcommand{\be}{\ensuremath{\mathbb E}}
\newcommand{\N}{\ensuremath{\mathbb N}}
\newcommand{\Z}{\ensuremath{\mathbb Z}}
\newcommand{\R}{\ensuremath{\mathbb R}}
\newcommand{\muhaar}{\ensuremath{\mu_{\rm Haar}} }
\newcommand{\G}{\ensuremath{\mathscr G}}
\newcommand{\ghat}{\ensuremath{\mathscr{G}^\wedge}}
\newcommand{\E}{\ensuremath{\mathscr E}}
\newcommand{\B}{\ensuremath{\mathscr B}}
\newcommand{\A}{\ensuremath{\mathscr A}}
\newcommand{\calp}{\ensuremath{\mathcal P}}
\newcommand{\U}{\ensuremath{\mathcal U}}
\newcommand{\heart}[2]{\ensuremath{\left\| #1 \right\|_{\psi,#2}}}
\newcommand{\normo}[1]{\ensuremath{\left\| #1 \right\|^1_\varphi}}
\newcommand{\normi}[1]{\ensuremath{\left\| #1 \right\|^\infty_\zeta}}
\newcommand{\normp}[1]{\ensuremath{\left\| #1 \right\|^\times_\phi}}
\newcommand{\comment}[1]{}

\newtheorem{theorem}[equation]{Theorem}
\newtheorem{prop}[equation]{Proposition}

\newtheorem{lemma}[equation]{Lemma}

\newtheorem{utheorem}{\textrm{\textbf{Theorem}}}

\theoremstyle{definition}
\newtheorem{remark}[equation]{Remark}
\newtheorem{definition}[equation]{Definition}

\numberwithin{equation}{section}

\begin{document}
\title{Integration and measures on the space of countable labelled
graphs}

\author[Apoorva Khare and Bala Rajaratnam]{Apoorva Khare and Bala
Rajaratnam\\ Stanford University}
\address{Department of Mathematics, Stanford University, Stanford, CA -
94305, USA}
\email{A.K.: khare@stanford.edu; B.R.: brajarat@stanford.edu}


\date{\today}

\begin{abstract}
In this paper we develop a rigorous foundation for the study of
integration and measures on the space $\G(V)$ of all graphs defined on a
countable labelled vertex set $V$. We first study several interrelated
$\sigma$-algebras and a large family of probability measures on graph
space. We then focus on a ``dyadic" Hamming distance function
$\heart{\cdot}{2}$, which was very useful in the study of differentiation
on $\G(V)$. The function $\heart{\cdot}{2}$ is shown to be a Haar
measure-preserving bijection from the subset of infinite graphs to the
circle (with the Haar/Lebesgue measure), thereby naturally identifying
the two spaces. As a consequence, we establish a ``change of variables"
formula that enables the transfer of the Riemann-Lebesgue theory on
$\R$ to graph space $\G(V)$. This also complements previous work in which
a theory of Newton-Leibnitz differentiation was transferred from the real
line to $\G(V)$ for countable $V$. Finally, we identify the Pontryagin
dual of $\G(V)$, and characterize the positive definite functions on
$\G(V)$.
\end{abstract}
\maketitle

\section{Introduction and main results}

The study of very large graphs and their limits has recently been the
focus of tremendous interest, given its importance in a variety of
scientific disciplines including probability and statistics,
combinatorics, computer science, machine learning, and network analysis
in various applied fields. In this regard several limiting theories have
been developed in the literature. Prominent among these is the
comprehensive theory of graphons, which are limits of (dense) unlabelled
graphs (see \cite{Lo} and the references therein).

In the present paper, we work in the parallel setting of labelled graphs
and their limits. Our motivation comes from the fact that often graphs in
real-world situations and observed network data are labelled, and each
vertex has a specific meaning. Similarly in theoretical probability such
as Markov random fields and their applications, nodes in graphs represent
variables that are not exchangeable owing to the dependencies in the
underlying model. This provides motivation to study the space of labelled
graphs and their limits.

In \cite{KR1} a framework was introduced in which to study all finite
labelled graphs at once; namely, the space $\G(V)$ of graphs with a fixed
countable, labelled vertex set $V$. The algebraic and topological
properties of $\G(V)$, as well as continuous functions on $\G(V)$, were
studied in \cite{KR1}. Moreover, a theory of differentiation on graph
space was developed in \cite{KR1}; see also \cite{DGKR} for
differentiation in the unlabelled setting in graphon space.
Note also that the space of graphons is naturally equipped with a large
family of measures that arise from sampling.
We now explore the parallel setting of labelled graph space, with the aim
of studying measures on $\G(V)$ and developing a theory of integration.
This is the goal of the present paper.

The space $\G(V)$ is a compact abelian group; hence the associated Haar
measure naturally gives rise to a theory of integration.
Our first objective is to identify and study the Haar measure.
The next goal of this work is to explore the connections between Haar
integration on graph space and the Riemann-Lebesgue theory on $\R$. As a
consequence of our investigations, we show below that integration on
$\G(V)$ can be reduced to that on the unit interval. This is akin to
\cite{KR1}, in which differentiation on $\G(V)$ was shown to be closely
related to the one-variable Newton-Leibnitz theory on $\R$.\medskip

In this section we will state the main results of the paper.
We begin by setting some notation.

\begin{definition}
Given a fixed labelled set $V$, define the corresponding (labelled) graph
space $\G(V)$ to be the set of all graphs with vertex set $V$. In other
words, $\G(V) = \{ 0, 1 \}^{K_V} = (\Z / 2 \Z)^{K_V}$, where $K_V$ is the
complete graph on $V$.
Also define $\G_0(V)$ to be the set of all graphs with finitely many
edges, and $\G_1(V)$ to be the set of all \textit{co-finite} graphs --
i.e., the complements in $K_V$ of finite graphs.
\end{definition}

Henceforth every labelled graph with vertex set $V$ will be identified
with its edge set, which is a subset of $K_V$. Note that $\G(V)$ is the
set of all functions $f : K_V \to \Z / 2\Z$, the discrete field with two
elements. This makes $\G(V)$ a commutative topological $\Z / 2\Z$-algebra
under pointwise addition and multiplication. In particular, the binary
operation $G + G' := G \Delta G'$ makes $\G(V)$ into an abelian
topological group, where $\Delta$ denotes the symmetric difference and
the zero element ${\bf 0} \in \G(V)$ is given by the empty graph on $V$.
Note also that $\G(V)$ is 2-torsion, i.e., $G+G = {\bf 0}$ for all $G \in
\G(V)$.

We next discuss the topological structure of labelled graph space $\G(V)$
for a countable vertex set $V$, as studied in \cite{KR1}. 
The following family of metrics on $\G(V)$ was crucially used in
\cite{KR1} in developing differential calculus in $\G(V)$, and is also
important for the purposes of the present paper.

\begin{definition}\label{D2}
Suppose $V$ is countable and $\psi : K_V \to \N$ is a fixed bijection.
Given $a > 1$, define $d_{\varphi_a} : \G(V) \times \G(V) \to [0,\infty)$
and $\heart{\cdot}{a} : \G(V) \to [0,\infty)$ via:
\begin{equation}
d_{\varphi_a}(G,G') := \sum_{e \in G \Delta G'} a^{-\psi(e)}, \qquad
\heart{G}{a} := d_{\varphi_a}({\bf 0},G).
\end{equation}

\noindent Next, a sequence $\{ G_n : n \in \N \}$ in $\G(V)$ is said to
\textit{converge (to $G \in \G(V)$)} if the indicator sequences $\{ {\bf
1}_{e \in G_n} : n \in \N \}$ each converge (to ${\bf 1}_{e \in G}$) for
each edge $e \in K_V$.
\end{definition}

The following theorem collects together some of the topological results
in \cite{KR1} on labelled graph space $\G(V)$, which are needed for the
purposes of this paper.

\begin{theorem}[\cite{KR1}]\label{Theart2}
For a fixed labelled set $V$, the set $\G(V)$ of graphs is a commutative,
totally disconnected, compact Hausdorff topological $\Z/2\Z$-algebra.

Now suppose that $V$ is countable and $\psi : K_V \to \N$ is a fixed
bijection.
\begin{enumerate}
\item The maps $\{ d_{\varphi_a} : a \geq 1 \}$ are translation-invariant
metrics on $\G(V)$, which are all topologically equivalent and metrize
the above notion of graph convergence (i.e., they generate the product
topology on $\G(V) = \{ 0, 1 \}^{K_V}$).

\item The sets $\G_0(V), \G_1(V)$ are dense in $\G(V)$.

\item The map $2 d_{\varphi_3}({\bf 0},-) = 2 \heart{\cdot}{3} : \G(V)
\to [0,1]$ is a homeomorphism onto the Cantor set. Thus $\G(V)$ is a
compact metric space.

\item The map $d_{\varphi_2}({\bf 0},-) = \heart{\cdot}{2} : \G(V) \to
[0,1]$ is a surjection, which is a bijection outside $\G_0(V)$. For every
finite nonempty graph $G$, there exists a unique co-finite graph $G'$
such that $\heart{G}{2} = \heart{G'}{2}$.
\end{enumerate}
\end{theorem}

We now present the main results in this paper. Since $\G(V)$ is a compact
abelian group, it is natural to seek out its associated Haar measure. We
identify this measure in our first main result. We also show that the
Haar measure is intimately connected to the distinguished metric
$\heart{\cdot}{2}$ (for any $\psi$) that was used in \cite{KR1} to
develop a differential calculus on $\G(V)$. More precisely, the following
holds.

\begin{utheorem}\label{T1}
Fix a countably infinite set $V$. The Haar measure $\muhaar$ on $\G(V)$
is the unique probability measure $\mu_{1/2}$ induced from the
Bernoulli$(\frac{1}{2})$-measure on each factor $\{ 0, 1 \}$ of $\G(V)$.
Now given any bijection $\psi : K_V \to \N$, the Haar measure of any
open or closed $\heart{\cdot}{2}$-ball (in $\G(V)$) of radius $\epsilon
\in [0,1]$ is $\epsilon$.
\end{utheorem}

Note here that the Borel $\sigma$-algebra $\B_{\G(V)}$ of $\G(V)$, as well as
the Haar measure $\mu_{1/2}$, do not depend on the choice of labelling
$\psi : K_V \to \N$.


It is natural to ask if the measure space $\G(V)$ with its Borel
$\sigma$-algebra, can be modelled by a more familiar probability space.
(This is akin to Theorem \ref{Theart2}, which provided familiar
topological models for graph space.) Note moreover that the last
assertion in Theorem \ref{T1} has an obvious analogue for the usual
Lebesgue measure, which is in fact the Haar measure on the real line. It
is now natural to ask if the two Haar measures are related. The following
result answers both of these questions, and shows how to transfer
integration from $\G(V)$ to $\R$.

\begin{utheorem}\label{Thaar}
Fix a bijection $\psi : K_V \to \N$.
\begin{enumerate}
\item The map $\heart{\cdot}{2} : \G(V) \to [0,1]$ -- or to the circle
$S^1 = \R / \Z$ -- is a measurable, Haar measure-preserving map that is a
bijection outside the countable (measure zero) sets $\G_0(V), \G_1(V)$.

\item Suppose $f : [0,1] \to [-\infty, \infty]$ is Lebesgue integrable.
Then,
\[ \be_{\muhaar}[f(\heart{\cdot}{2})] = \int_0^1 f(x)\ dx. \]
Conversely, for all integrable $g : \G(V) \to [-\infty, \infty]$, we
have
\[ \be_{\muhaar}[g] = \int_0^1 g( (\heart{\cdot}{2})^{-1}(x))\ dx. \]
\end{enumerate}
\end{utheorem}

\noindent Thus, Haar integration can be carried out on labelled graph
space by transferring the classical Lebesgue theory from the unit
interval (or the circle) to $\G(V)$.

The remaining sections are devoted to proving the above results. We add
moreover that additional results concerning Fourier analysis, the
Pontryagin dual, and positive definite functions for $\G(V)$ are
shown in Section \ref{Spontryagin} below.

\section{Measures on graph space}\label{Smeas}

In this section we develop the necessary tools required to show Theorem
\ref{T1}. We begin by studying several $\sigma$-algebras on graph space
and showing how they are related. We then study probability measures on
$\G(V)$ and prove Theorem \ref{T1}.

\subsection{$\sigma$-algebras on graph space}

We begin with an arbitrary (fixed) labelled index set $V$
of vertices. Let $V(e)$ denote the vertices attached to an edge $e \in
K_V$; then $\G(V) = \times_{e \in K_V} \calp(K_{V(e)})$ is the Cartesian
product of power sets, and each set is a $\sigma$-algebra of size 2.
Define the product $\sigma$-algebra $\Sigma_{\rm meas}$ on $\G(V)$ to be
the $\sigma$-algebra generated by the cylinder sets
\begin{equation}\label{Eproduct}
S_{e_0} := \times_{e \neq e_0} \{ \emptyset, \{ e \} \} \times \{ e_0
\}, \qquad e_0 \in K_V.
\end{equation}

We now define several other $\sigma$-algebras on $\G(V)$, as well as a
closely related family of sets.

\begin{definition}
Given disjoint subsets $I_0, I_1 \subset K_V$, define
\begin{equation}
\E(I_0, I_1) := \{ G \in \G(V) : I_1 \subset G,\ I_0 \subset K_V
\setminus G \}.
\end{equation}

\noindent Now define the following $\sigma$-algebras on $\G(V)$:
\begin{itemize}
\item $\B_{\G(V)}$ is the \textit{Borel $\sigma$-algebra}, generated by
all open sets.

\item $\Sigma_0$ is the $\sigma$-algebra generated by all compact sets.

\item $\Sigma_\E$ is the $\sigma$-algebra generated by all sets $\E(I_0,
I_1)$ for disjoint $I_0, I_1 \subset K_V$.

\item $\Sigma_{\E,0}$ is the $\sigma$-algebra generated by all $\E(I_0,
I_1)$ for finite (or countable) disjoint $I_0, I_1 \subset K_V$.
\end{itemize}
\end{definition}

Consider the case when $V$, and hence $K_V$, is countable. In this case,
as shown in \cite{KR1}, the product topology on the compact space $\G(V)$
can be metrized; this yields another candidate $\sigma$-algebra, as
described presently. The main result of this subsection relates all of
the above $\sigma$-algebras.

\begin{theorem}\label{Tsigma}
Suppose $V$ is countable, and $d$ is a metric on $\G(V)$ that metrizes
the product topology. Define $\Sigma_{d,ball}$ to be the $\sigma$-algebra
generated by the open (or closed) $d$-balls in $\G(V)$. Then,
\[ \Sigma_{\rm meas} = \Sigma_{\E,0} = \Sigma_\E = \Sigma_{d,ball} =
\B_{\G(V)} = \Sigma_0. \]
\end{theorem}

In order to prove Theorem \ref{Tsigma}, some preliminary results are
needed. The first result collects some basic facts about the sets
$\E(I_0, I_1)$.

\begin{lemma}\label{L4}
Suppose $V$ is arbitrary and $I_0, I_1, J_0, J_1 \subset K_V$ are all
disjoint.
\begin{enumerate}
\item Then one has:
\begin{equation}\label{Emeasure6}
\E(I_0 \cup J_0, I_1 \cup J_1) = \E(I_0, J_0) \cap \E(I_1, J_1).
\end{equation}

\noindent In particular, $\E(-,-)$ is inclusion-reversing in each
argument.

\item Given $S \subset \G(V) \ni G$, define $S + G := \{ G' + G : G' \in
S \}$. Then,
\begin{equation}\label{Emeasure1}
\E(I_0, I_1) = \E(I_0 \cup I_1, \emptyset) + I_1,
\end{equation}

\noindent where for every $I \subset K_V$, $\E(I, \emptyset)$ is an ideal
of (the $\Z / 2\Z$-algebra) $\G(V)$.

\item Given any $G \in \G(V)$,
\begin{equation}\label{Emeasure2}
\E(I_0, I_1) + G = \E \left( (I_0 \setminus G) \coprod (I_1 \cap G), (I_1
\setminus G) \coprod (I_0 \cap G) \right).
\end{equation}

\item Given disjoint sets $I_0, I_1 \subset K_V$, the set $\E(I_0, I_1)$
is closed in $K_V$. It is open if and only if $I_0 \coprod I_1$ is
finite.
\end{enumerate}
\end{lemma}

\begin{proof}
All but the last part are easy to prove using the definitions. For the
last part, note that for finite disjoint $I_0, I_1 \subset K_V$, $\E(I_0,
I_1)$ is closed as well as open in $\G(V)$. Hence $\E(I_0, I_1) \subset
\G(V)$ is closed for all disjoint $I_0, I_1 \subset K_V$, by using
Equation \eqref{Emeasure6}. Finally, suppose that $I_0 \cup I_1$ is
infinite; the goal is now to prove that its complement in $(\Z / 2
\Z)^{K_V}$ is not closed in the product topology.
To do so, it suffices to produce a sequence $G_n \notin \E(I_0, I_1)$,
that converges to a graph $G_0 \in \E(I_0, I_1)$. Thus, fix a countable
subset $\{ i_n : n \in \N \} \subset I_0 \cup I_1$, and define
$G_n := I_1 \Delta \{ i_n \}, \ G_0 := I_1 \in \E(I_0, I_1)$.
It is easy to check that $G_n \to G_0$ in $\G(V)$, and that this sequence
satisfies the desired properties.
\end{proof}

In order to state and prove the next result, the following notation is
required.

\begin{definition}\label{Dnpsi}
Suppose $V$ is countable. Fix a bijection $\psi : K_V \to \N$ and define
$E_n(\psi) := \{ e \in K_V : \psi(e) \leq n \}$. Given $a>1, \epsilon
\geq 0$, and $G \in \G(V)$, define $B(G,\epsilon,\heart{\cdot}{a})$ to be
the open ball in $(\G(V), d_{\varphi_a})$ with center $G$ and radius
$\epsilon$, and $\overline{B}(G,\epsilon,\heart{\cdot}{2})$ to be its
closure in $(\G(V), d_{\varphi_a})$.
\end{definition}

The last preliminary result shows that the sets $\E(I_0,I_1)$ lie in the
Borel $\sigma$-algebra for finite $I_0, I_1$.

\begin{prop}\label{P4}
Suppose $V$ is countable. Fix a bijection $\psi : K_V \to \N$, and given
disjoint subsets $I_0, I_1 \subset \N$, define $\E(I_0, I_1) :=
\E(\psi^{-1}(I_0), \psi^{-1}(I_1))$.
\begin{enumerate}
\item If $I_0 \coprod I_1 = \{ 1, \dots, n \}$ for some $n$, then
\begin{eqnarray*}
\E(I_0, I_1) & = & B(\psi^{-1}(I_1), 2^{-n}, \heart{\cdot}{2}) \coprod \{
K_V \setminus \psi^{-1}(I_0) \}\\
& = & B(\psi^{-1}(I_1), 2^{-n}, \heart{\cdot}{2}) \bigcup B(K_V \setminus
\psi^{-1}(I_0), 2^{-n}, \heart{\cdot}{2})\\
& = & \overline{B}(\psi^{-1}(I_1), 2^{-n-1}, \heart{\cdot}{2}) \bigcup
\overline{B}(K_V \setminus \psi^{-1}(I_0), 2^{-n-1}, \heart{\cdot}{2}).
\end{eqnarray*}

\item For all finite disjoint $I_0, I_1$, $\E(I_0, I_1)$ is a finite
union of open balls, as well as closed balls. Alternatively, it can be
partitioned into finitely many open balls and a finite set.
\end{enumerate}
\end{prop}

\begin{proof}\hfill
\begin{enumerate}
\item For the first equality, it is clear that the left-hand side is
contained in the right. To show the reverse inclusion, if $G$ is in the
right-hand side, then $G \Delta \psi^{-1}(I_1)$ cannot intersect
$E_n(\psi)$ (which was defined in Definition \ref{Dnpsi}), so it must be
disjoint from $\psi^{-1}(I_0)$, and must contain $\psi^{-1}(I_1)$.
The second equality is now easy to show, and one inclusion in the third
equality as well. For the converse, if $G \in \E(I_0, I_1)$, then either
$\psi^{-1}(n+1) \in G$ -- whence $G$ is in the second closed ball --
otherwise $G$ is in the first closed ball.

\item Suppose $\max (I_0 \coprod I_1) = n$, and $\{ 1, \dots, n \}
\setminus (I_0 \cup I_1) = \{ m_1 < \cdots < m_l \}$. Then:
\[ \E(I_0, I_1) = \coprod_{J_0 \subset \{ m_1, \dots, m_l \}} \E(I_0 \cup
J_0, I_1 \cup \{ m_1, \dots, m_l \} \setminus J_0). \]

\noindent The result now follows from the previous part.
\end{enumerate}
\end{proof}

Finally, we use the above results to prove the main result in this
subsection.

\begin{proof}[Proof of Theorem \ref{Tsigma}]
It is clear by Definition \ref{D2} that every $G \in \G(V)$ is the limit
of a sequence $G_n$ of finite graphs: set $G_n := G \cap E_n(\psi)$. Now
\comment{
if $d$ metrizes the product topology in $\G(V)$, then $d(G_n,G) \to
0$, which shows that $\G_0(V)$ is dense in $(\G(V),d)$ for all such $d$.
It is now standard that the countable collection $\{ B(G, 2^{-n}, d) : n
\in \N, G \in \G_0(V) \}$ of open balls is a base for $\G(V)$. Therefore
}
$\B_{\G(V)} \subset \Sigma_{d,ball}$; the reverse inclusion is obvious.
Also note that the $\sigma$-algebras generated by the open and closed
$d$-balls are both equal.

We now claim that some of the inclusion relations hold among the
$\sigma$-algebras defined above, for arbitrary vertex sets $V$. Namely,
we claim for all sets $V$:
\begin{equation}\label{Esigma}
\Sigma_{\rm meas} = \Sigma_{\E,0} \subset \Sigma_\E \subset \B_{\G(V)} =
\Sigma_0.
\end{equation}

To show Equation \eqref{Esigma}, note that since $\G(V)$ is a compact
Hausdorff topological space by Theorem \ref{Theart2}, hence $K \subset
\G(V)$ is compact if and only if $\G(V) \setminus K$ is open. This proves
that $\B_{\G(V)} = \Sigma_0$. Also note that $\Sigma_{\E,0}$ is generated
by all sets $\E(I_0, I_1)$, where we may assume $I_0, I_1$ to be either
finite or countable -- that both of these choices yield equivalent
$\sigma$-algebras follows from repeated applications of Equation
\eqref{Emeasure6}.
Next, note that $S_{i_0} = \E(\emptyset, \{ i_0 \})$ for all $i_0 \in
I$. Hence if $I_0, I_1$ are finite disjoint subsets of $K_V$, then
\[ \E(I_0, I_1) = \bigcap_{i \in I_0} (\G(V) \setminus S_i) \cap
\bigcap_{i \in I_1} S_i. \]

\noindent This proves that $\Sigma_{\rm meas} = \Sigma_{\E,0}$.
Next, that $\Sigma_{\E,0} \subset \Sigma_\E$ is obvious. Finally,
$\Sigma_\E \subset \B_{\G(V)}$ by Lemma \ref{L4}, since the sets $\E(I_0,
I_1)$ are closed for disjoint $I_0, I_1 \subset K_V$.

Given Equation \eqref{Esigma}, it remains to prove that every open set
$\U \subset \G(V)$ is in $\Sigma_\E$. First, fix any $G \in \U$ that is
not cofinite, i.e., $G \notin \G_1(V)$. Since $\U$ is open, $B(G,
\epsilon, \heart{\cdot}{2}) \subset \U$ for some $\epsilon>0$.
Choose $N > 0$ such that $2^{-N} < \epsilon$, and fix $n > N$ such that
$\psi^{-1}(n) \notin G$ (since $G \notin \G_1(V)$). Define the finite
graph $G_0 = E_{n-1}(\psi) \cap G \in \G_0(V)$. Then since $n \notin
\psi(G)$ and $G \notin \G_1(V)$, hence $\heart{G - G_0}{2} < 2^{-n} <
2^{-N} < \epsilon$. Thus, $2^{-n} \leq 2^{-N-1} < \epsilon/2$.

Now use the first part of Proposition \ref{P4} with $I_1 := \{ \psi(G_0)
\} \subset \{ 1, \dots, n-1 \}$. It follows that $B(G_0, 2^{-n},
\heart{\cdot}{2}) \in \Sigma_\E$. Moreover, $2^{-n} < \epsilon/2$, so
\[ G \in B(G_0, 2^{-n}, \heart{\cdot}{2}) \subset B(G_0, \epsilon/2,
\heart{\cdot}{2}) \subset B(G, \epsilon, \heart{\cdot}{2}) \subset \U. \]

\noindent But now we are done: $\U$ is the union of the countable set $\U
\cap \G_1(V)$, and for each $G \in \U \setminus \G_1(V)$, the open ball
$B(G_0, 2^{-n}, \heart{\cdot}{2})$ as above.
Since each of these sets is in $\Sigma_\E$, and there are only countably
many such sets (since they are in bijection with a subset of $\G_0(V)
\times \N$), hence $\U$ is a countable union of elements of $\Sigma_\E$.
Thus, $\U \in \Sigma_\E$, whence $\B_{\G(V)} \subset \Sigma_\E$, as
desired.
\end{proof}

\subsection{Haar measure}\label{Shaar}

We now define and study a large family of measures $\mu_P$ on the space
$\G(V)$, eventually focussing on the Haar measure and the proof of the
first main result, Theorem \ref{T1}. Given any labelled set $V$ and any
function $P : K_V \to [0,1]$, the map $\mu_{P,e}$ assigning $P(e)$ to $\{
e \}$ and $1 - P(e)$ to $\emptyset$ is a Bernoulli probability measure on
the Bernoulli space $K_{V(e)}$. Now recall the $\sigma$-algebra
$\Sigma_{\rm meas}$ \eqref{Eproduct}, and define the product measure
$\mu_P$ on finite intersections of these sets via: $\mu_P(S_{e_1} \cap
\cdots \cap S_{e_n}) := \prod_{i=1}^n P(e_i)$ for distinct $e_i \in K_V$.
One can ask if this information is sufficient to determine $\mu_P$ on
$(\G(V), \Sigma_{\rm meas})$. To answer this question, recall the
following results.

\begin{prop}[{\cite[Theorem 6.1 and Corollary 6.1]{JP}}]\label{Pjp}
Suppose a $\sigma$-algebra $(\Omega,\A)$ is generated by a subset
$\mathcal{C} \subset \A$ that is closed under finite intersections.
\begin{enumerate}
\item Two probability measures on $\A$ are equal if and only if they
agree on $\mathcal{C}$.

\item Suppose $\mathcal{C}$ is an algebra, and $\mu' : \mathcal{C} \to
[0,1]$ is a probability measure (satisfying countable additivity as well
as that $\mu'(\Omega) = 1$). Then $\mu'$ extends uniquely to a measure on
all of $\A$.
\end{enumerate}
\end{prop}

The proof uses the Monotone Class Theorem \cite[Theorem 6.2]{JP}. In
particular, the result affirmatively answers the above question. Thus
$\mu_P$ satisfies the following properties when $V$ is countable:
\begin{itemize}
\item $\mu_P$ is determined uniquely by its restriction to finite
intersections of the sets $S_{e_0}$. In particular, $\mu_P$ is a
probability measure on $\G(V)$, and one writes: $\mu_P := \prod_{e \in
K_V} \mu_{P,e}$.
\item For all disjoint $I_0, I_1$, the sets $\E(I_0, I_1)$ and all Borel
sets are $\mu_P$-measurable.
\item In particular, every locally constant function on $\G(V) \setminus
C$ (where $C$ is a countable set) is $\mu_P$-measurable for all $P : K_V
\to [0,1]$.
\end{itemize}

\begin{remark}
A special case is the measure $\mu_p$ for $p \in [0,1]$, given by $P(e) =
p\ \forall e$. This is precisely the \textbf{Erd\"os-R\'enyi model} for
$\G(V)$.
When $V$ is countable, this construction generalizes the analysis in
\cite{Cam}, where $\G(V)$ is identified (via a bijection $\psi : K_V \to
\N$) with $2^{\N} = \{ 0, 1 \}^{\N}$, the space of binary sequences, as
well as with $[0,1]$ via the binary expansion of any real number $x \in
[0,1]$. Note that this map is precisely the function $\heart{\cdot}{2}$.
Cameron \cite{Cam} also informally writes down a measure on $\G(V)$ as
being induced from countably many independent tosses of a fair coin; this
is precisely the Erd\"os-R\'enyi measure $\mu_{1/2}$ above.
\end{remark}

We now outline the contents in the remainder of the paper. The immediate
task is to prove the first main result in the paper (Theorem \ref{T1})
using the above preliminary results. Following the proof, in Section
\ref{Scomput} we compute expectations of several real-valued functions on
graph space with respect to the Erd\"os-R\'enyi measures $\mu_p$, as an
illustration of how to work with these measures. Having computed the
Haar-expectations of specific functions, we then prove the other main
result, Theorem \ref{Thaar} (in the following section); this result deals
with transporting the Haar-expectation of arbitrary functions between
graph space and the real line. The paper concludes with a study of
Fourier analysis on $\G(V)$.

\begin{proof}[Proof of Theorem \ref{T1}]
We make the following claim:\medskip

\textit{For any $V$ (and up to scaling), when restricted to $\Sigma_{\rm
meas} = \Sigma_{\E,0}$, the Haar measure necessarily equals $\mu_P$ with
$P(e) = 1/2$ for all $e \in K_V$. In other words, $\muhaar \equiv
\mu_{1/2}$ on $\Sigma_{\rm meas}$.}\medskip

%

That $\muhaar$ exists and is the unique translation-invariant probability
measure on $\G(V)$ follows by a classical result of Weil \cite{Weil}
(also proved by Cartan), since $\G(V)$ is a compact topological group.
Now suppose $I \subset K_V$ is finite. Then for all partitions $I = I_0
\coprod I_1$, one computes using Equation
\eqref{Emeasure1}:
\[ \muhaar(\E(I_0, I_1)) = \muhaar(\E(I_0, I_1) + I_1) =
\muhaar(\E(I,\emptyset)), \]

\noindent by translation-invariance. Now since $\displaystyle \G(V) =
\coprod_{I_0 \subset I} \E(I_0, I \setminus I_0)$, hence
\begin{equation}\label{Ecylinder}
\muhaar(\E(I_0, I_1)) = 2^{-|I|} = \mu_{1/2}(\E(I_0, I_1)).
\end{equation}

\noindent Note that the sets $\E(I_0, I_1)$ generate $\Sigma_{\E,0}$.
Hence $\muhaar \equiv \mu_{1/2}$ on $\Sigma_{\E,0}$ by Proposition
\ref{Pjp}, since the collection of sets $\E(I_0, I_1)$ (for finite
disjoint $I_0, I_1 \subset K_V$) is closed under finite intersections by
Equation \eqref{Emeasure6}.

We now prove the main theorem using the claim. Note that computing
$\muhaar$ on $\Sigma_{\rm meas}$ uniquely determines the Haar measure
when $V$ is countable, since $\Sigma_{\E,0} = \B_{\G(V)}$ when $V$ is
countable (by Theorem \ref{Tsigma}). Thus $\muhaar \equiv \mu_{1/2}$.

Next, we assert that the Haar measure of a $\heart{\cdot}{2}$-ball of
radius $\epsilon \in [0,1]$ is $\epsilon$. The assertion is clear for
$\epsilon = 0$ (i.e.~for the empty set) and $\epsilon = 1$ (which yields
the entire space $\G(V)$ except one point), since points must have Haar
measure zero or else $\muhaar(\G(V)) = \infty$.
Now assume $\epsilon \in (0,1)$ and $G = {\bf 0}$ (by the
translation-invariance of $\mu_{1/2}$). It is enough to prove the
assertion for $\epsilon$ of the form $2^{-n_1} + 2^{-n_2} + \cdots +
2^{-n_k}$ with $n_1 < n_2 < \cdots < n_k \in \N$, because then given any
$\epsilon > 0$, approximate it from below by a nondecreasing sequence
$\epsilon_n \to \epsilon^-$ (and from above by a nonincreasing sequence
$\epsilon'_n \to \epsilon^+$), with each $\epsilon_n, \epsilon'_n$ a
finite sum of the above form. (For instance, take $\epsilon_n$ to be the
truncated binary expansions of $\epsilon$.) Then,
\begin{align*}
\mu_{1/2}(B({\bf 0}, \epsilon, \heart{\cdot}{2})) \geq \mu_{1/2}
\left( \bigcup_{n=1}^\infty B({\bf 0}, \epsilon_n, \heart{\cdot}{2})
\right) = &\ \lim_{n \to \infty} \mu_{1/2}(B({\bf 0}, \epsilon_n,
\heart{\cdot}{2}))\\
= &\ \lim_{n \to \infty} \epsilon_n = \epsilon,
\end{align*}

\noindent and similarly, $\mu_{1/2}(B({\bf 0}, \epsilon,
\heart{\cdot}{2})) \leq \lim_{n \to \infty} \epsilon'_n = \epsilon$.

Thus it remains to prove the assertion for $\epsilon = 2^{-n_1} +
2^{-n_2} + \cdots + 2^{-n_k}$; we do so by induction on $k \geq 0$. For
$k=0$ the result was proved earlier in this proof; from this the result
follows for $k=1$ by using Proposition \ref{P4} and Equation
\eqref{Ecylinder}. Now given the result for $k-1 \geq 0$, set $\epsilon'
= \sum_{0<i<k} 2^{-n_i}$ and $\epsilon = \epsilon' + 2^{-n_k}$. The
graphs in $B({\bf 0}, \epsilon, \heart{\cdot}{2})$ have possible
$\heart{\cdot}{2}$-values in $[0, \epsilon) = [0, \epsilon') \coprod \{
\epsilon' \} \coprod (\epsilon', \epsilon)$.
Using that the binary expansion is a bijection from $[0,1]$ to binary
sequences (except on a countable set described in Theorem \ref{Theart2}),
one notes that the graphs corresponding to the first two sets of
$\heart{\cdot}{2}$-values above are, respectively, $S_{\epsilon'} :=
B({\bf 0}, \epsilon', \heart{\cdot}{2})$ and the doubleton set
$S' := \{ T_{k-1}, T_{k-2} \coprod \{ n_{k-1} + 1, n_{k-1} + 2, \dots \}
\}$, where $T_k := \{ n_1, \dots, n_k \}$ for all $k$.
Moreover, if $\heart{G}{2} \in (\epsilon', \epsilon)$ for some $G \in
\G(V)$, then it is not too hard to show (again using binary expansions,
via Theorem \ref{Theart2}) that from among the integers $1, \dots, n_k$,
the only ones in $\psi(G)$ are precisely $n_1, \dots, n_{k-1}$. Thus, the
graphs whose $\heart{\cdot}{2}$-values lie in $(\epsilon', \epsilon)$
form the set
\[ S'' := \E(\{ 1, \dots, n_k \} \setminus T_{k-1}, T_{k-1}) \setminus \{
\{ n_1, \dots, n_{k-1}, n_k + 1, n_k + 2, \dots \} \}. \]

\noindent Since points have zero measure, $\mu_{1/2}(S'') = 2^{-n_k}$ by
Equation \eqref{Ecylinder}. Hence:
\begin{align*}
\mu_{1/2}(B({\bf 0}, \epsilon, \heart{\cdot}{2})) =
\mu_{1/2} \left( S_{\epsilon'} \coprod S' \coprod S'' \right) =
 \epsilon' + 0 + 2^{-n_k} = \epsilon,
\end{align*}

\noindent where the penultimate equality follows from the induction
hypothesis and previous results. This completes the proof for open balls,
by induction.

For the closed ball $\overline{B} := \overline{B}(G, \epsilon,
\heart{\cdot}{2})$, if $\epsilon = 1$ then $\mu_{1/2}(\overline{B}) =
\mu_{1/2}(\G(V)) = 1$, while if $\epsilon < 1$, 
\[ B(G, \epsilon, \heart{\cdot}{2}) \subset \overline{B}(G, \epsilon,
\heart{\cdot}{2}) = \bigcap_{n \in \N} B(G, \epsilon + n^{-1},
\heart{\cdot}{2}). \]

\noindent Thus the result for closed balls follows from the result for
open balls, since we have:
\[ \epsilon \leq \mu_{1/2}(\overline{B}(G, \epsilon, \heart{\cdot}{2}))
\leq \inf_{(1 - \epsilon)^{-1} \leq n \in \N} \epsilon + n^{-1} =
\epsilon. \]
\end{proof}

\subsection{Examples: computing expectations}\label{Scomput}

We now work out some examples of computing expectations with respect to
the probability measures on $\G(V)$ that were introduced in Section
\ref{Shaar}. In the following results, $V$ is assumed to be countable.

\begin{prop}\label{P6}
Fix $P : K_V \to [0,1]$ and a bijection $\psi : K_V \to \N$.
Then every countable set has $\mu_P$-measure zero if and only if
$\displaystyle \prod_{e \in K_V} \max(P(e), 1-P(e)) = 0$. In particular,
this holds if there exists $\epsilon>0$ such that the set $\{ e \in K_V :
P(e) \in (\epsilon, 1 - \epsilon) \}$ is infinite (e.g., if $\mu_P \equiv
\mu_p$ for $p \in (0,1)$).
\end{prop}

\begin{proof}
Given $G \in \G(V)$ and $n \in \N$, first define
\[ p_n(G) := \prod_{e \in E_n(\psi) \cap G} P(e) \prod_{e \in E_n(\psi)
\setminus G} (1 - P(e)), \quad \pi_n := \prod_{e \in E_n(\psi)}
\max(P(e), 1 - P(e)). \]

\noindent Note that both $p_n(G)$ and $\pi_n$ are non-increasing
nonnegative sequences, with $p_0(G) = \pi_0 = 1$. In particular, they are
both convergent. Moreover, note that
\[ \{ G \} = \bigcap_{n \in \N} \E(\{ 1, \dots, n \} \setminus G, G \cap
\{ 1, \dots, n \})\ \forall G \in \G(V). \]

\noindent Now since $\displaystyle \mu_P(\E(I_0, I_1)) = \prod_{i \in
I_0} (1 - P(\psi^{-1}(i))) \prod_{i \in I_1} P(\psi^{-1}(i))$ for finite
disjoint subsets $I_0, I_1 \subset \N$, we compute:
\begin{align*}
0 \leq &\ \mu_P(G) = \lim_{n \to \infty} \mu_P(\E(\{ 1, \dots, n \}
\setminus G, G \cap \{ 1, \dots, n \})) = \lim_{n \to \infty} p_n(G) \\
\leq &\ \lim_{n \to \infty} \pi_n = \mu_P(G_P),
\end{align*}

\noindent where $G_P$ (or its set of edges) equals $\{ e \in K_V : P(e)
\geq 1/2 \}$. Thus, every countable set has $\mu_P$-measure zero, if and
only if $\mu_P(G) = 0\ \forall G$, if and only if $\mu_P(G_P) = 0$.

Finally, for the second sub-part, we simply note that infinitely many of
the terms in the product are less than $\epsilon < 1$, so $\mu_P(G_P) =
0$.
\end{proof}

In order to state the next result, first recall some notation from
\cite{KR1}.

\begin{definition}[\cite{KR1}]
Let $V$ be a countable labelled set.
\begin{enumerate}
\item Define $\displaystyle \ell_+^1(K_V) := \Big\{ \varphi : K_V \to (0,
\infty)\ \big| \sum_{e \in K_V} \varphi(e) < \infty \Big\}$.

\item Also define $\ell_+^\infty(K_V)$ to be the set of functions $\zeta
: K_V \to (0, \infty)$ such that $\zeta(K_V)$ has precisely one
accumulation point: $0$ (and $\infty$ is not an accumulation point, i.e.,
$\zeta$ is bounded).

\item Further define $\ell^1_\times(K_V) := \{ \phi : K_V \to (1,
\infty),\ \prod_{e \in K_V} \phi(e) < \infty \}$.

\item Given ${\bf 0} \neq G \in \G(V)$, $\varphi \in \ell_+^1(K_V)$,
$\zeta \in \ell_+^\infty(K_V)$, and $\phi \in \ell^1_\times(K_V)$,
define:
\[ \normo{G} := \sum_{e \in G} \varphi(e), \quad \normi{G} := \max \{
\zeta(e) : e \in G \}, \quad \normo{{\bf 0}} = \normi{{\bf 0}} =
\normp{{\bf 0}} := 0, \]

\noindent and $\normp{G} := \left( \prod_{e \in G} \phi(e) - 1
\right)^{1/n}$, with the smallest $n \in \N$ such that $2^n \geq 1 +
\prod_{e \in K_V} \phi(e)$.
\end{enumerate}
\end{definition}

It was shown in \cite{KR1} that the maps $\normo{.}, \normi{.}$ induce
topologically equivalent translation-invariant metrics on graph space
$\G(V)$, which metrize its product topology. One can similarly show the
following fact.

\begin{lemma}
For all $\phi \in \ell^1_\times(K_V)$, the maps $\normp{.}$ induce
topologically equivalent translation-invariant metrics on $\G(V)$, which
metrize its product topology.
\end{lemma}

\noindent The following result provides examples of computing the
expectations of the aforementioned functions (and others) with respect to
the Erd\"os-R\'enyi type product measures $\mu_p$ defined in Section
\ref{Shaar}.

\begin{prop}\label{P2}
For all $p \in (0,1)$, the expectation $\be_{\mu_p}$ is a linear
functional on the space of measurable functions $h : (\G(V), \B_{\G(V)})
\to \R$.
\begin{enumerate}
\item Given $k>0$ and a graph $G \in \G(V)$ with at least $k$ edges,
labelled by $\psi(G) = \{ n_1 < n_2 < \cdots < n_k < \cdots \} \subset
\N$, define the \textit{$k$th minimum edge number} of $G$ to be
$\Psi_k(G) := n_k$. Then $\be_{\mu_p}[\Psi_k] = k/p$.

\item Given $\zeta \in \ell_+^\infty(K_V)$, choose any bijection
$\psi_\zeta : K_V \to \N$ as in Lemma \ref{Lstat} (below). Then for all
$f : {\rm im}(\zeta) \to \R$,
\[ \be_{\mu_p}[f(\normi{.})] = \sum_{n \in \N} p(1-p)^{n-1}
f(\zeta(\psi_\zeta^{-1}(n))). \]

\item For all $\varphi \in \ell_+^1(K_V)$ and $G \in \G(V)$, $\normo{G} =
\sum_{e \in K_V} \varphi(e) {\bf 1}_{G \in \E(\emptyset, \psi(e))}$.
Moreover,
\[ \be_{\mu_p}[\normo{.}] = p \normo{K_V}, \qquad
\be_{\mu_p}[(\normo{.})^2] =  p(1-p) \|K_V\|^1_{\varphi^2} + p^2
(\normo{K_V})^2. \]

\item If $\normp{K_V} \leq \sqrt[n]{2^n - 2}$ for some $\phi \in
\ell^1_\times$ and $n>0$, then
\[ \be_{\mu_p}[(\normp{.})^n] = -1 + \prod_{e \in K_V} (1 -p + p \phi(e))
\leq (\normp{K_V})^n. \]
\end{enumerate}
\end{prop}

\noindent In particular, if $X : \G(V) \to \R$ denotes the random
variable $X(G) := \normo{G}$, then $X$ has $\mu_p$-mean $p \normo{K_V}$,
and variance $p(1-p) \|K_V\|^1_{\varphi^2}$. More generally, one can
imitate the proof below to show that for all $n \in \N$,
$\be_{\mu_p}[(\normo{.})^n] = \be_{\mu_p}[X^n]$ equals some
``homogeneous" polynomial in $\{ \|K_V\|^1_{\varphi^r} =
\|K_V\|^r_\varphi : 0 \leq r \leq n \}$, with coefficients that are
polynomials in $p$.
(Here, ``homogeneous" means that every monomial has the same total
degree, with $\|K_V\|^1_{\varphi^r}$ having degree $r$.)

Also note that some of these results can be shown more generally for all
$\mu_P$ (with $P : K_V \to [0,1]$ as in Section \ref{Shaar}). For
example, if $\normp{K_V} \leq \sqrt[n]{2^n - 2}$, then
\[ \be_{\mu_P}[\normo{.}] = \sum_{e \in K_V} P(e) \varphi(e), \qquad
\be_{\mu_P}[(\normp{.})^n] = -1 + \prod_{e \in K_V} (1 - P(e) + P(e)
\phi(e)). \]

The following observation will be used to prove Proposition \ref{P2}.

\begin{lemma}\label{Lstat}
For all $\zeta \in \ell_+^\infty(K_V)$, there exists a bijection
$\psi =\psi_\zeta : K_V \to \N$, such that $\zeta(\psi^{-1}(1)) \geq
\zeta(\psi^{-1}(2)) \geq \cdots$.
\end{lemma}

\noindent For instance if $\zeta(e) = \heart{e}{a}$ for fixed $a>1$ and
all $e \in K_V$, then
$\zeta(\psi^{-1}(n)) = a^{-n}$, so $\psi_\zeta = \psi$.

\begin{proof}
Since the only accumulation point of the image set $\zeta(K_V)$ is
$0$, it follows that for every $e \in K_V$, there are only finitely many
values above $\zeta(e)$ -- and they can all be totally ordered. In other
words, every subset of $\zeta(K_V)$ has a maximum element.
Thus, define $\beta : \N \to K_V$ inductively: $\beta(1)$ is any element
of $\displaystyle \arg \max_{e \in K_V} \zeta(e)$, and given $\beta(1),
\dots, \beta(k-1)$, define $\beta(k)$ to be any element of the set
$\displaystyle \arg \max_{e \in S_k} \zeta(e)$, where
$S_k := K_V \setminus \{ \beta(1), \dots, \beta(k-1) \}$.
It is clear that this inductively covers all $e \in K_V$, by the previous
paragraph. Hence $\beta : \N \to K_V$ is a bijection such that
$\zeta(\beta(1)) \geq \zeta(\beta(2)) \geq \cdots$. Now define $\psi =
\psi_\zeta := \beta^{-1}$.
\end{proof}
\newpage

\begin{proof}[Proof of Proposition \ref{P2}]\hfill
\begin{enumerate}
\item Note that $\Psi_k$ is defined at all but countably many graphs in
$\G(V)$. Moreover, the $k$th minimum edge of $G \in \G(V)$ is $n$ if and
only if $n \geq k$ and $\psi(G) \cap \{ 1, \dots, n-1 \}$ has size
exactly $k-1$. This means that $G \in \E(\{ 1, \dots, n-1 \} \setminus S,
S \coprod \{ n \})$, where $S \subset \{ 1, \dots, n-1 \}$ has size
precisely $k-1$. Now there are precisely $\binom{n-1}{k-1}$ such sets
$S$, and by Proposition \ref{P6}, each corresponding set $\E (\{ 1,
\dots, n-1 \} \setminus S, S \coprod \{ n \})$ has measure $(1-p)^{n-k}
p^k$.
Hence we exclude the countable (measure zero) set of graphs with fewer
than $k$ edges, and compute:
\[ \be_{\mu_p}[\Psi_k] = \sum_{n=k}^\infty n \cdot \binom{n-1}{k-1} p^k
(1-p)^{n-k} = k p^k \sum_{n=k}^\infty \binom{n}{k} (1-p)^{n-k}. \]

\noindent On the other hand, the Binomial Formula easily yields:
\begin{align*}
p^{-(k+1)} = &\ \sum_{l=0}^{\infty} \binom{-(k+1)}{l} (-(1-p))^l\\
= &\ \sum_{l=0}^{\infty} (-1)^l \binom{(k+1) + l - 1}{l} (-(1-p))^l =
\sum_{l=0}^{\infty} \binom{k+l}{k} (1-p)^l.
\end{align*}

\noindent Setting $l = n-k$, the expected value above equals $k p^k \cdot
p^{-(k+1)} = k/p$.

\item Note that $\zeta(\psi_\zeta^{-1}(1)) \geq \zeta(\psi_\zeta^{-1}(2))
\geq \cdots$ by choice of $\psi_\zeta$. It is clear that $\normi{G} =
\zeta(\psi_\zeta^{-1}(n))$ if $G \in \E_\zeta(\{ 1, \dots, n-1 \}, \{ n
\})$, with $\E_\zeta$ denoting the $\E$-set corresponding to
$\psi_\zeta$. Since these sets partition $\G(V) \setminus \{ {\bf 0} \}$,
use Theorem \ref{T1} to compute:
\begin{align*}
\be_{\mu_p}[f(\normi{.})] = &\ \sum_{n \in \N}
f(\zeta(\psi_\zeta^{-1}(n))) \mu_p(\E_\zeta(\{ 1, \dots, n-1 \}, \{ n
\}))\\
= &\ \sum_{n \in \N} p (1-p)^{n-1} f(\zeta(\psi_\zeta^{-1}(n))).
\end{align*}

\item Define $f_N(G) := \sum_{n=1}^N \varphi(\psi^{-1}(n)) {\bf 1}_{G \in
\E(\emptyset, \{ n \})}$. Thus, $\{ f_N \}$ is a nondecreasing sequence
of $[0, \infty)$ valued $\mu$-measurable functions on $\G(V)$. It is not
hard to show that their pointwise limit at any $G \in \G(V)$ is
\[ \sum_{n \in \N} \varphi(\psi^{-1}(n)) {\bf 1}_{G \in \E(\emptyset, \{
n \})} = \sum_{e \in K_V} \varphi(e) {\bf 1}_{G \in \E(\emptyset, \{
\psi(e) \})} \leq \sum_{e \in K_V} \varphi(e) = \normo{K_V} < \infty. \]

\noindent Moreover, in computing $\normo{G}$, $\varphi(e)$ is a summand
if and only if $e \in G$, i.e., $G \in \E(\emptyset, \{ \psi(e) \})$.
This proves the first statement.
Now use Proposition \ref{P6} and the Monotone Convergence Theorem to
compute:
\begin{align*}
\be_{\mu_p}[\normo{G}] = &\ \lim_{N \to \infty} \be_{\mu_p}[f_N] =
\lim_{N \to \infty} \sum_{n=1}^N \varphi(\psi^{-1}(n))
\mu_p(\E(\emptyset, \{ n \}))\\
= &\ \sum_{n=1}^\infty \varphi(\psi^{-1}(n)) \cdot p = p \normo{K_V},
\end{align*}


\noindent which proves the first part of the second statement.
For the second part, note that $(\normo{G})^2 = \lim_{N \to \infty}
f_N^2$. Now write out the summand:
\begin{align*}
f_N^2(G) = &\ \sum_{n=1}^N \varphi(\psi^{-1}(n))^2 {\bf 1}_{G \in
\E(\emptyset, \{ n \})}\\
&\ + 2 \sum_{n=1}^N \sum_{m=1}^{n-1} \varphi(\psi^{-1}(m))
\varphi(\psi^{-1}(n)) {\bf 1}_{G \in \E(\emptyset, \{ m \})} {\bf 1}_{G
\in \E(\emptyset, \{ n \})}.
\end{align*}

\noindent Taking expectations yields:
\begin{eqnarray*}
\be_{\mu_p}[f_N^2(G)] & = & p \sum_{n=1}^N \varphi(\psi^{-1}(n))^2 + 2
p^2 \sum_{n=1}^N \sum_{m=1}^{n-1} \varphi(\psi^{-1}(m))
\varphi(\psi^{-1}(n))\\
& = & p(1-p) \sum_{n=1}^N \varphi(\psi^{-1}(n))^2 + \left( p \sum_{n=1}^N
\varphi(\psi^{-1}(n)) \right)^2.
\end{eqnarray*}

\noindent From above computations, the second term is just
$\be_{\mu_p}[f_N]^2$, which converges to $(p \normo{K_V})^2$ by the first
part. Hence as above, using Proposition \ref{P6} and the Monotone
Convergence Theorem, $\be_{\mu_p}[(\normo{G})^2]$ equals
$\displaystyle \lim_{N \to \infty} \be_{\mu_p}[f_N^2(G)] = p(1-p)
\|K_V\|^1_{\varphi^2} + p^2 (\normo{K_V})^2$.

\item This is similar to the previous part: define
\[ f_N(G) := \prod_{n=1}^N \left( 1 + (\phi(\psi^{-1}(n)) - 1) {\bf 1}_{G
\in \E(\emptyset, \{ n \})} \right). \]

\noindent Once again, $0 \leq f_N(G) \leq f_{N + 1}(G) \leq \prod_{e \in
K_V} \phi(e) < \infty$ for all $G$, so we can apply the Monotone
Convergence Theorem. Moreover, the pointwise limit of the $f_N$ is
precisely $1 + (\normp{.})^n$, and one easily checks that the expectation
of any product of $k$ distinct indicators as above is $p^k$. This proves
that
$\displaystyle
\be_{\mu_p}[1 + (\normp{.})^n] = \lim_{N \to \infty} \prod_{n=1}^N (1 +
p(\phi(\psi^{-1}(n)) - 1))$, 
and the result follows.
\end{enumerate}
\end{proof}

\section{Haar integration and Fourier analysis}

We now study graph space $\G(V)$ in further detail. Recall that $\G(V)$
is a compact topological group; these are objects for which a
comprehensive theory of analysis and probability has been systematically
developed in the literature -- see e.g.~\cite{Gre,P,Ru}. In this section
we further explore two aspects of the theory: first, we find a more
familiar model for graph space as a compact group with Haar measure.
Second, we study Fourier analysis on $\G(V)$. This includes classifying
the Pontryagin dual, as well as all positive definite functions on
$\G(V)$.

\subsection{Haar integration on graph space}

In this part we study the relationship between the Haar measure on
$\G(V)$ and the Lebesgue measure on $\R$. As seen above, the Haar measure
of an $\epsilon$-ball is $\epsilon$. This property also holds (up to
scaling by 2) for the usual Lebesgue measure $\mu_\R$ on the real line.
Thus, it is natural to ask if the two measure spaces $(\G(V), \muhaar =
\mu_{1/2})$, and $(\R, \mu_\R)$ -- or more precisely, the circle group
$S^1$ with its Haar measure -- are related. If so, how does one account
for the ``Jacobian" in transforming Haar integration from $\G(V)$ into
the usual Lebesgue theory on $\R$? These questions are the focus of the
next main result in the paper.

\begin{proof}[Proof of Theorem \ref{Thaar}]
Note that $\heart{\cdot}{2}$ is a continuous map $: \G(V) \to \R$, hence
(Borel) measurable with respect to the respective Borel
$\sigma$-algebras. Thus, consider two measures on the Borel
$\sigma$-algebra $\B_{\G(V)}$, given by $A \mapsto \mu_{1/2}(A)$ and $A
\mapsto \mu_\R(\heart{A}{2})$. The latter is indeed a measure that
satisfies countable additivity because $\heart{\cdot}{2}$ is a bijection
except on a countable set, and countable sets have measure zero in either
measure.

We now claim that the measures $\mu_{1/2}$ and $\mu_{\R} \circ
\heart{\cdot}{2}$ on $(\G(V), \B_{\G(V)})$ agree on all sets $\E(I_0,
I_1)$ for finite disjoint $ I_0, I_1 \subset \N$. If this holds, then
since these sets are closed under finite intersections, Proposition
\ref{Pjp} implies that these measures are identical (using Theorem
\ref{Tsigma}). This proves the first assertion.

To prove the claim, recall from \cite{Cam} that $([0,1], \mu_\R)$ is
equivalent -- via binary expansion -- to the countable sequences of
independent tosses of a fair coin.
Moreover, countable sets have probability zero in all cases, and the set
$\E(I_0, I_1)$ corresponds precisely to all sequences where the $i$th
coin toss is a tail if $i \in I_0$, and a head if $i \in I_1$. In turn,
these correspond to the set of all $x \in [0,1]$ whose $i$th digit in the
binary expansion is $0$ if $i \in I_0$ and $1$ if $i \in I_1$ -- and
these sets are measurable because they are unions of intervals.
It is now clear, by partitioning $[0,1]$ into $2^{\max(I_0 \cup
I_1)}$-many intervals of equal length, that each of these sets has
measure $2^{-|I_0 \cup I_1|}$, which proves the claim.

We now show the second assertion.
Note that $f(x) = \pm \infty$ only on a set of measure zero, since $f$ is
Lebesgue integrable. Next, the functions $f^\pm := \max(\pm f, 0)$ are
also measurable (and integrable) if $f$ is; hence by linearity it
suffices to prove the result for each of them. Thus,
suppose without loss of generality that $0 \leq f < \infty$.
We carry out a standard construction to approximate $f$ by a sequence of
nonnegative simple functions $0 \leq f_1 \leq f_2 \leq \cdots$ on
$[0,1]$, which converge pointwise to $f$ almost everywhere. Given $n \in
\N$, define $I_{n,k} := \left[ \frac{k-1}{2^n}, \frac{k}{2^n} \right)$
for $1 \leq k \leq 2^{2n}$, and $I_{n, 2^{2n} + 1} := [2^n, \infty)$. Now
define $A_{n,k} := f^{-1}(I_{n,k}) \subset [0,1]$, and $B_{n,k} := \{ G
\in \G(V) : f(\heart{G}{2}) \in I_{n,k} \}$. Thus by Theorem \ref{Thaar},
both $A_{n,k}$ and $B_{n,k}$ are measurable and of equal measures. Now
define the functions
\[ f_n := \sum_{k=1}^{2^{2n} + 1} \frac{k-1}{2^n} {\bf 1}_{A_{n,k}},
\qquad g_n := \sum_{k=1}^{2^{2n} + 1} \frac{k-1}{2^n} {\bf 1}_{B_{n,k}}.
\]

\noindent It is then standard that $0 \leq f_n \leq f_{n+1} \leq f$ at
each point, and $f_n(x) \to f(x)$ for all $x$. The same facts also hold
for $g_n$ and $g$, where we define: $g(G) := f(\heart{G}{2})$ and $g_n(G)
:= f_n(\heart{G}{2})$.
%
Moreover, since $\heart{\cdot}{2}$ is measure-preserving, the simple
functions $f_n, g_n$ are pullbacks of each other (via the invertible map
$\heart{\cdot}{2}$ and its inverse -- outside the countable set
$\G_0(V)$, say). Hence,
\[ \be_\mu[g_n] = \int_{\G(V)} g_n(G)\ d\mu_{1/2} = \int_0^1 f_n(x)\ dx =
\int_0^1 f_n\ d \mu_\R\ \forall n \in \N. \]

\noindent Now use the Monotone Convergence Theorem twice:
\[ \be_{\mu_{1/2}} [f(\heart{\cdot}{2})] = \be_{\mu_{1/2}} [g] = \lim_{n
\to \infty} \be_{\mu_{1/2}} [g_n] = \lim_{n \to \infty} \int_0^1 f_n(x)\
dx = \int_0^1 f(x)\ dx. \]
\end{proof}

\begin{remark}\label{R3}
Note that $\heart{\cdot}{2}$ is not a bijection on $\G(V)$; thus to make
sense of $(\heart{\cdot}{2})^{-1}$, ignore all finite graphs $\G_0(V)$,
and/or cofinite graphs $\G_1(V)$ (since all countable sets have measure
zero). Then $\heart{\cdot}{2}$ is a measure-preserving bijection on the
complement, whence a random variable $T|_{\G(V) \setminus \G_0(V)}$ is
measurable if and only if $T \circ (\heart{\cdot}{2})^{-1} : [0,1]
\setminus \heart{\G_0(V)}{2} \to \G(V) \setminus \G_0(V) \to X$ is
measurable.
\end{remark}

\begin{remark}
We now illustrate an application of Theorem \ref{Thaar}. Recall the
definition of the ``$k$th minimum edge number" as defined in Proposition
\ref{P2}(1). One can now show that
\[ \Psi_1(G) = \min \psi(G) := - \lfloor \log_2 \heart{G}{2} \rfloor. \]

\noindent More generally, $\Psi_k(G)$ can be inductively defined as
$f_k(\heart{G}{2})$, where:
\[ f_k(x) := f_1 \left( x - \sum_{i=1}^{k-1} 2^{-f_i(x)} \right), \qquad
f_1(x) := - \lfloor \log_2 x \rfloor. \]

\noindent In particular, one can compute the expected value of $\Psi_1$
(with respect to $\mu_{1/2}$) using Theorem \ref{Thaar}, to be
$\displaystyle -\int_0^1 \lfloor \log_2(x) \rfloor\ dx$, which can be
shown to converge to $2$. Recall that this expectation was also computed
in Proposition \ref{P2}(1).
\end{remark}

\subsection{Pontryagin duality and Walsh-Rademacher
functions}\label{Spontryagin}

We conclude the paper with a discussion of Fourier analysis in graph
space. We require the following terminology.

\begin{definition}
Suppose $V$ is any fixed vertex set.
\begin{enumerate}
\item A function $f : \G(V) \to \mathbb{C}$ is {\em positive definite} if
for all integers $n \in \N$ and $G_1, \dots, G_n \in \G(V)$, the matrix
$(f(G_i \Delta G_j))_{1 \leq i, j \leq n}$ is positive semidefinite.

\item Given a finite graph $E \in \G_0(V)$, the corresponding {\em Walsh
function} $\chi_E : \G(V) \to \mathbb{C}$ is defined as follows:
\[ \chi_E(G) := (-1)^{|E \cap G|} = \prod_{e \in E} \chi_{\{ e \}}(G). \]
\end{enumerate}
\end{definition}

The Walsh functions turn out to be important for several reasons,
including for Fourier analysis via Pontryagin duality. Recall the
following terminology:

\begin{definition}
A unitary character of a group $\G$ is a group homomorphism $\chi : \G
\to S^1$, the unit circle in $\mathbb{C}^\times$. The {\em Pontryagin
dual} of a locally compact abelian group is simply the set of continuous
unitary characters, which form a group under pointwise multiplication.
\end{definition}

\noindent Note for $\G = \G(V)$ that all unitary characters have image in
$\{ \pm 1 \}$, since $\G(V)$ is a group with exponent 2.
Now the following result completely characterizes all positive definite
functions on $\G(V)$, as well as its Pontryagin dual.

\begin{theorem}\label{Tpont}
Suppose $V$ is a countable set. Then the Pontryagin dual to $\G(V)$ is
naturally identified with its subgroup $\G_0(V)$ of finite graphs, via
Walsh functions. They also form an orthonormal basis of $L^2(\G(V),\R)$.

Moreover, a function $f : \G(V) \to \mathbb{C}$ is positive definite and
satisfies $f({\bf 0}) = 1$, if and only if there exists a probability
measure $\mu$ on $\G_0(V)$ (i.e., a countable set of nonnegative numbers
$\mu(H)$ that add up to $1$), such that
\[ f(G) = \sum_{H \in \G_0(V)} (-1)^{|G \cap H|} \mu(H). \]
\end{theorem}

\noindent (In particular, $f$ has image in $[-1,1]$.)

Since $\G(V)$ is a compact abelian group, one can also apply the theory
of Pontryagin duality to carry out Fourier analysis on it, or to state
Parseval's identity and Plancherel's theorem (a useful reference is
\cite[Chapter 1]{Ru}). We now write down some of the results in this
setting.

\begin{prop}\label{Pwalsh}
Suppose $V$ is an arbitrary set (of labelled vertices).
\begin{enumerate}
\item The ``group algebra" $L^1(\G(V),\R)$ is a Banach algebra under
convolution.

\item The set of Walsh functions $\{ \chi_E : E \in \G_0(V) \}$ is an
orthonormal subset of $L^2(\G(V),\R)$. When $V$ is countable, the Walsh
functions form a complete/Hilbert basis; moreover, they transform into
the usual Walsh functions -- i.e., products of {\em Rademacher functions}
-- via the Haar measure-preserving map $\heart{\cdot}{2}$.
\end{enumerate}
\end{prop}

\noindent Note that the first part follows from \cite[Theorem 1.1.7]{Ru},
and the second part follows from Theorem \ref{Thaar}, since the Walsh
functions form a complete orthonormal system in $L^2([0,1],\R)$.
Moreover, they comprise the Pontryagin dual of $\G(V)$:

\begin{theorem}\label{Tplancherel}
Suppose $\ghat$ is the Pontryagin dual group to $\G(V)$ for a set $V$.
\begin{enumerate}
\item For all finite sets $E \in \G_0(V)$, we have $\chi_E \in \ghat$,
with image in $\{ \pm 1 \}$.

\item $\ghat$ is a discrete (locally compact) abelian group, which is
metrizable if $V$ is countable.

\item \emph{(Plancherel's Theorem.)} The Fourier transform, when
restricted to $(L^1 \cap L^2)(\G(V))$, is a linear isometry in the
$L^2$-metric, onto a dense subset of $L^2(\ghat)$. Hence it has a unique
extension to a unitary operator from $L^2(\G(V))$ onto $L^2(\ghat)$ (for
some Haar measure $\mu^\wedge$ on $\ghat$).

\item When $V$ is countable, $\ghat$ is precisely the set of Walsh
functions, and the assignment $\chi_E \mapsto E := \{ e \in K_V :
\chi_E(\{ e \}) = -1 \}$ is a group isomorphism onto $(\G_0(V), \Delta)$.
\end{enumerate}
\end{theorem}

\begin{proof}
To show (1), one shows that $\chi_{\{ e \}}$ is continuous for each $e
\in K_V$:
\[ \chi_{\{ e \}}^{-1}(-1) = \E(\emptyset, \{ e \}), \qquad \chi_{\{e
\}}^{-1}(1) = \E(\{ e \}, \emptyset), \]

\noindent and these are both open by Proposition \ref{P4}. Note that by
\cite[Theorem 1.2.5]{Ru}, $\ghat$ is discrete since $\G(V)$ is compact.
Similarly, $\ghat$ is metrizable since $\G(V)$ is separable. This proves
(2). Part (3) is shown (for more general $\G$) in \cite[Theorem
1.6.1]{Ru}. Part (4) is also not hard to show -- see e.g.~\cite{Fi,Wa}.
\end{proof}

Finally, we prove the remaining unproved result above.

\begin{proof}[Proof of Theorem \ref{Tpont}]
All but the last assertion follow from Proposition \ref{Pwalsh} and
Theorem \ref{Tplancherel}. To prove the last part, apply Bochner's
Theorem \cite[Theorem 1.4.3]{Ru} to the compact abelian group $\G(V)$.
Thus, every normalized function $f$ is of the form $\displaystyle f(G) =
\int_{\xi \in \ghat} \xi(G)\ d \mu(\xi)$ for some probability measure
$\mu$ on $\ghat$. Since $\ghat \cong \G_0(V)$ from above, every measure
is a countable tuple as claimed.
\end{proof}

\subsection*{Concluding remarks}

In this paper we analyzed measures and integration on labelled graph space
$\G(V)$. 
We showed that $\G(V)$ with its Haar measure is very closely related to
the circle with its Haar measure, which allowed us to transport
Haar-Lebesgue integration on $[0,1]$ over to graph space $\G(V)$.

A more involved task is to study random graphs -- i.e., sequences of
$\G(V)$-valued random variables. This involves the analysis of measurable
functions from a probability space into $\G(V)$ (as opposed to
real-valued functions of graphs studied in this paper). Note that graph
space $\G(V)$ is a 2-torsion group, and hence does not embed as a group
into a normed linear space. Thus the next step in the study of labelled
graphs and their limits involves developing the foundations of
probability theory on $\G(V)$, and studying probability inequalities and
stochastic convergence on random graphs. The study of probability theory
on graph space is addressed in forthcoming work \cite{KR3}. Such a
formalism is essential in order to discuss issues like probability
generating mechanisms for graphs, or to sample from graph space.

\subsection*{Acknowledgments}

We thank David Montague and Doug Sparks for useful discussions.



\end{document}